\documentclass[a4paper,11pt]{article}
\usepackage[english]{babel}
\usepackage{titling}
\usepackage[latin1]{inputenc}
\usepackage{latexsym}
\usepackage{graphicx} %extended version vom package graphics
\usepackage{amsmath,amsfonts,amssymb,amsthm,url,textcomp}
\usepackage{amsfonts}
\usepackage{csquotes}
\usepackage{a4wide}

\newtheorem{theorem}{Theorem}
\newtheorem{definition}[theorem]{Definition}
\newtheorem{lemma}[theorem]{Lemma}

\theoremstyle{plain}

\newtheorem*{propposition}{Proposition}

\theoremstyle{remark}

\newtheorem*{remmark}{Remark}

\begin{document}

\title{On the dynamics of endomorphisms of finite groups}

\author{Alexander Bors\thanks{The author is supported by the Austrian Science Fund (FWF):
Project F5504-N26, which is a part of the Special Research Program \enquote{Quasi-Monte Carlo Methods: Theory and Applications}. \newline 2010 \emph{Mathematics Subject Classification}: 05C38, 05C60, 05C76, 05E15, 20D45, 20D60, 37P99. \newline \emph{Key words and phrases:} Finite dynamical system, finite group, group endomorphisms, state space}}

\date{}

\maketitle

\begin{abstract}
Aiming at a better understanding of finite groups as finite dynamical systems, we show that by a version of Fitting's Lemma for groups, each state space of an endomorphism of a finite group is a graph tensor product of a finite directed $1$-tree whose cycle is a loop with a disjoint union of cycles, generalizing results of Hern{\'a}ndez-Toledo on linear finite dynamical systems, and we fully characterize the possible forms of state spaces of nilpotent endomorphisms via their \enquote{ramification behavior}. Finally, as an application, we will count the isomorphism types of state spaces of endomorphisms of finite cyclic groups in general, extending results of Hern{\'a}ndez-Toledo on primary cyclic groups of odd order.
\end{abstract}

\section{Some background}

Finite dynamical systems have recently gained a lot of interest not only within mathematics, but also for their practical applications in areas such as cryptography, pseudorandom number generation and reverse engineering. For example, one approach to study gene regulatory networks is to discretize both the data and the time flow and then work in a finite dynamical system of the form $(k^n,f)$, where $k$ is a finite field and $f$ a (polynomial) map $k^n\rightarrow k^n$, see \cite{JLSS07a}. Such so-called polynomial finite dynamical systems are also objects of current theoretical research, and there are still many open questions.

However, there is a well-established theory of so-called \textit{linear finite dynamical systems} (a special case, abbreviated henceforth by LFDSs). These consist of a finite-dimensional vector space $V$ over a finite field together with a linear map $f:V\rightarrow V$. The first results (written in the language of circuit theory) are on the case where $f$ is an automorphism of $V$ and are due to Elspas from 1959, see \cite{Els59a}. Much later, in 2005, Hern{\'a}ndez-Toledo extended these results to LFDSs in general, see \cite{Her05a}. The results give strong restrictions on the possible forms of state spaces compared to arbitrary finite dynamical systems, for example, all state spaces of LFDSs are graph tensor products of a $1$-tree whose cycle is a loop (representing the nilpotent part of $f$, $\mathrm{nil}(f)$) with a disjoint union of cycles (representing the periodic part of $f$, $\mathrm{per}(f)$), and $V$ decomposes as a direct sum of $\mathrm{nil}(f)$ and $\mathrm{per}(f)$.

In this paper, we generalize the results on LFDSs, but going in a different direction than usually: What if not $f$ is replaced by a more complicated polynomial map, but we keep the \enquote{nice} property of $f$ being an endomorphism and instead replace the vector space structure on the underlying set by a group structure (note that any vector space endomorphism is in particular an endomorphism of the underlying additive group)? It turns out that the basic results on LFDSs mentioned in the last paragraph can be transferred to this more general situation. A first indication of this fact can be found in the 2012 paper \cite{Sha12a}, where Sha shows that state spaces of endomorphisms of finite cyclic groups are graph tensor products as described above. Also, as we will see, the group laws impose strong restrictions on the form of the $1$-tree representing the nilpotent part.

\section{Results on the structure of the state space}

Let us first fix some notation and terminology. We denote by $\mathbb{N}$ the set of natural numbers (including $0$) and by $\mathbb{N}^+$ the set of positive integers. As usual, a \textit{finite dynamical system} (FDS) is a pair $(X,f)$ where $X$ is a finite set (whose elements will be referred to as \textit{points}) and $f$ a so-called \textit{endofunction of $X$}, that is, a function $X\rightarrow X$. For $n\in\mathbb{N}$, $f^n$ denotes the $n$-th iteration of $f$ (i.e., the $n$-th power of $f$ in the monoid of endofunctions of $X$). Points $x$ such that, for some positive $n$, $f^n(x)=x$ are called \textit{periodic}, and the smallest such $n$ is called the \textit{period of $x$}. Points which are not periodic are called \textit{transient}. For any transient point $y$, there exists a least positive integer $h$ such that $f^h(y)$ is periodic; this $h$ is called the \textit{height of $y$}, denoted $\mathrm{ht}(y)$. Following the terminology in \cite{LP01a}, the \textit{state space} of $(X,f)$, denoted $\Gamma_f$, is the digraph with vertex set $X$ which has a directed edge from $x$ to $y$ if and only if $y=f(x)$. By the general theory of FDSs, $\Gamma_f$ always is a directed $1$-forest. For FDSs $(X,f),(Y,g)$, the FDS $(X\times Y,f\times g)$, where $f\times g$ is the endofunction of $X\times Y$ mapping $(x,y)\mapsto(f(x),g(y))$, is called the \textit{product of $(X,f)$ and $(Y,g)$}. Observe that $\Gamma_{f\times g}=\Gamma_f\times\Gamma_g$, where the $\times$ on the RHS denotes the graph (tensor) product. An \textit{isomorphism between FDSs $(X,f)$ and $(Y,g)$} is a bijection $\alpha:X\rightarrow Y$ such that $\alpha\circ f=g\circ\alpha$. It is easy to see that $\alpha$ is an isomorphism between $(X,f)$ and $(Y,g)$ if and only if $\alpha$ is an isomorphism between the state spaces $\Gamma_f$ and $\Gamma_g$.

From now on, we will always consider the situation where $X$ is a finite group $G$ (or, more precisely, its underlying set) and $f$ is a group endomorphism $\varphi$ of $G$; such FDSs will be referred to as \textit{finite dynamical groups} (FDGs). We set $\mathrm{nil}(\varphi):=\{g\in G\mid\exists n\in\mathbb{N}^+:\varphi^n(x)=1\}$ and define $\mathrm{per}(\varphi)$ as the set of periodic points of $\varphi$. As in the case of LFDSs, $\mathrm{nil}(\varphi)$ will be called the \textit{nilpotent part} and $\mathrm{per}(\varphi)$ the \textit{periodic part} of $\varphi$; if $\mathrm{nil}(\varphi)=G$, $\varphi$ is called \textit{nilpotent}. Note that by definition, $\mathrm{nil}(\varphi)$ is the union of the subsets $\mathrm{ker}^{(m)}(\varphi)$ of $G$ for $m\in\mathbb{N}$, where $\mathrm{ker}^{(m)}(\varphi)$, which we will call the \textit{$m$-th kernel of $\varphi$}, is just the $m$-th preimage of $\{1\}$ under $\varphi$. The result that, in case of an LFDS $(V,f)$, $V$ directly decomposes into $\mathrm{nil}(f)$ and $\mathrm{per}(f)$ generalizes to:

\begin{theorem}\label{composTheo}
Let $(G,\varphi)$ be an FDG. Then:

(1) $\mathrm{nil}(\varphi)$ is the largest subgroup of $G$ invariant under $\varphi$ on which the corresponding restriction of $\varphi$ is nilpotent. Also, $\mathrm{nil}(\varphi)$ is normal in $G$.

(2) $\mathrm{per}(\varphi)$ is the largest subgroup of $G$ invariant under $\varphi$ on which the corresponding restriction of $\varphi$ is an automorphism.

(3) $G=\mathrm{nil}(\varphi)\rtimes\mathrm{per}(\varphi)$.

(4) The FDS $(G,\varphi)$ is the product of the FDSs $(\mathrm{nil}(\varphi),\varphi_{|\mathrm{nil}(\varphi)})$ and $(\mathrm{per}(\varphi),\varphi_{|\mathrm{per}(\varphi)})$. In particular, $\Gamma_{\varphi}$ is the product of a $1$-tree whose cycle is a loop with a disjoint union of cycles.
\end{theorem}

\begin{proof}
For (1) and (2), note that it suffices to show that $\mathrm{nil}(\varphi)$ and $\mathrm{per}(\varphi)$ are subgroups (and $\mathrm{nil}(\varphi)$ normal), which is clear by observing that $\mathrm{nil}(\varphi)$ is the maximum (with respect to inclusion) of the ascending chain of normal subgroups $(\mathrm{ker}^{(m)}(\varphi))_{m\in\mathbb{N}}$ and $\mathrm{per}(\varphi)$ is the minimum of the descending chain of subgroups $(\mathrm{im}(\varphi^n))_{n\in\mathbb{N}}$.

From these observations, (3) immediately follows from the group version of Fitting's Lemma stated and proved as Theorem 4.2 in \cite{Car13a}, and (4) is clear by (3) and the structure of semidirect products.
\end{proof}

We now turn to the structure of the tree from the nilpotent part. First, some terminology:

\begin{definition}\label{rigidProcDef}
Let $\Gamma=(V,E)$ be a finite digraph, $v\in V$.

(1) A vertex $w\in V$ such that $(v,w)\in E$ is called a \textbf{successor} or \textbf{child} of $v$.

(2) The \textbf{procreation behavior of $v$} is the sequence $(a_k)_{k\in\mathbb{N}^+}$ such that for all positive integers $k$, $a_k$ is the number of children $c$ of $v$ such that there exists a directed path $(w_1,\ldots,w_k)$ in $\Gamma$ with $w_1=c$ (we say: $c$ \textbf{has (at least) $k-1$ successor generations} and call $a_k$ the \textbf{$k$-th procreation number of $v$}). For $n\in\mathbb{N}$, the \textbf{procreation behavior of length $n$ of $v$} is the $n$-tuple consisting of the first $n$ procreation numbers of $v$.

(3) We say that $\Gamma$ has \textbf{rigid procreation} if and only if for all $v,w\in V$ and all $n\in\mathbb{N}$ such that $v$ and $w$ both have $n$ successor generations, the procreation behaviors of length $n$ of $v$ and $w$ are equal.
\end{definition}

For any digraph $\Gamma=(V,E)$, the \textbf{dual digraph of $\Gamma$}, denoted $\Gamma^{\ast}$, is defined as $(V,E^{-1})$ with $E^{-1}$ the inverse relation of $E$, i.e., the set of all pairs $(y,x)$ such that $(x,y)\in E$.

\begin{theorem}\label{rigidProcTheo}
Let $(G,\varphi)$ be an FDG. Then $\Gamma_{\varphi}^{\ast}$ has rigid procreation.
\end{theorem}

\begin{proof}
First, note that periodic points of $\varphi$ have infinitely many successor generations in $\Gamma_{\varphi}^{\ast}$ and that it suffices to show that any point $v\in G$ which has $n$ successor generations has the same procreation behavior of length $n$ as $1_G$. This is clear for periodic points by the structure of $\Gamma_{\varphi}$ exhibited in Theorem \ref{composTheo}(4) which implies that the $k$-th procreation coefficient of any periodic point is $1$ plus the number of successors of $1_G$ in $\Gamma_{\varphi_{|\mathrm{nil}(\varphi)}}^{\ast}$ which have at least $k-1$ successor generations, so we can assume that $v$ is transient.

Fix any $w$ in the $n$-th successor generation of $v$ which does not appear in any earlier generation (in other words, $\mathrm{ht}(w)=\mathrm{ht}(v)+n$). First, we claim that any element in one of the $n$ successor generations of $v$ (including the element $v$ itself) has a unique representation of one of the forms $\varphi^k(w)\cdot x$ with $k\in\{0,\ldots,n\}$ and $x\in\mathrm{ker}^{(n-k)}(\varphi)$.

To see this, first take an element $g$ in one of the successor generations, say $\mathrm{ht}(g)=\mathrm{ht}(v)+n-k$ with $k\in\{0,\ldots,n\}$. It then follows that $\varphi^{n-k}(\varphi^k(w^{-1})g)=v^{-1}\cdot v=1_G$, so that indeed, $g$ can be written as $\varphi^k(w)\cdot x$ with $x:=\varphi^k(w^{-1})g\in\mathrm{ker}^{(n-k)}(\varphi)$. But also, clearly $\mathrm{ht}(\varphi^k(w)\cdot x)=\mathrm{ht}(v)+n-k$ if $x\in\mathrm{ker}^{(n-k)}(\varphi)$ so that $\varphi^k(w)\cdot x=\varphi^l(w)\cdot y$ first implies $k=l$ and then $x=y$.

We are now ready to show that the procreation behaviors of $v$ and $1_G$ of length $n$ coincide. Fix $k\in\{0,\ldots,n-1\}$, and let $x$ be a child of $1_G$ in $\Gamma_{\varphi}^{\ast}$ which has $k$ successor generations (i.e., contributes to the entry $a_k$ in the procreation behavior $(a_1,\ldots,a_n)$ of $1_G$). Then either $x=1$ or there exists $y\in\mathrm{ker}^{(k+1)}(\varphi)\setminus\mathrm{ker}^{(k)}(\varphi)$ such that $\varphi^k(y)=x$, in which case we readily check that $\varphi^{n-k-1}(w)\cdot y$ is an element in the $k$-th preimage of $\varphi^{n-1}(w)\cdot x$ under $\varphi$; summing up, we have an injection $x\mapsto \varphi^{n-1}(w)\cdot x$ from the set of children of $1_G$ with $k$ successor generations into the set of children of $v$ with $k$ successor generations.

But this is even a bijection, for if $\varphi^{n-1}(w)\cdot x$ is a child of $v$ which has $k$ successor generations and we fix an element $\varphi^{n-k-1}(w)\cdot y$ in the $k$-th preimage of $\varphi^{n-1}(w)\cdot x$ under $\varphi$, then we see immediately that $y$ must be in the $k$-th preimage of $x$ under $\varphi$ so that $x$ is a child of $1_G$ with $k$ successor generations. This proves the theorem.
\end{proof}

\begin{remmark}\label{isoRem}
Note that the isomorphism type of a state space of an endomorphism $\varphi$ of a finite group $G$ is completely determined by the procreation behavior of $1_G$ in $\Gamma_{\varphi}^{\ast}$ together with the orders $|\mathrm{per}_n(\varphi)|$ of the subgroups of $G$ consisting of periodic points whose period divides $n$ for the various $n\in\mathbb{N}^+$, a fact that we will frequently use without further reference when counting isomorphism types of state spaces of finite cyclic groups in the next section.
\end{remmark}

The last theorem on the state space structure of FDGs which we want to present here gives further information on the behavior of the procreation numbers of the identity element:

\begin{theorem}\label{stateGraphTheo}
Let $(G,\varphi)$ be an FDG and let $(a_k)_{k\in\mathbb{N}^+}$ be the procreation behavior of $1_G$ in $\Gamma_{\varphi}^{\ast}$. Then for all $k\in\mathbb{N}$, $a_1\cdots a_k=|\mathrm{ker}^{(k)}(\varphi)|$ (in particular, $a_k=[\mathrm{ker}^{(k)}(\varphi):\mathrm{ker}^{(k-1)}(\varphi)]$), and for all $n,m\in\mathbb{N}^+$, $n\leq m$ implies $a_m\mid a_n$.
\end{theorem}

\begin{proof}
The divisibity result is obtained by an application of Lagrange's theorem after some counting which will yield the first assertion as a \enquote{by-product}.

First, observe that in any finite digraph with rigid procreation, the number of endpoints of paths with length $r$ starting from some vertex $v$ with at least $r$ successor generations and procreation behavior of length $r$ equal to $(a_1,\ldots,a_r)$ is precisely $a_1\cdots a_r$, since by induction on $r$, the $a_r$ children of $v$ that have enough successor generations to contribute to this number each give $a_1\cdots a_{r-1}$ endpoints.

Applying this to the vertex $1_G$ in $\Gamma_{\varphi}^{\ast}$ yields $|\mathrm{ker}^{(k)}(\varphi)|=a_1\cdots a_k$. Now note that the $n$-th procreation number of $1_G$ in the dual of the state space of the FDG $(\mathrm{im}(\varphi),\varphi_{|\mathrm{im}(\varphi)})$ is the number of children of $1_G$ which have at least $k$ successor generations in $\Gamma_{\varphi}^{\ast}$, so the corresponding procreation behavior is given by the sequence $(a_{n+1})_{n\in\mathbb{N}^+}$, and we obtain $|\mathrm{ker}^{(k)}(\varphi)\cap\mathrm{im}(\varphi)|=a_2\cdots a_{k+1}$.

By Lagrange's Theorem, we now get $a_2\cdots a_{k+1}\mid a_1\cdots a_k$, that is, $a_{k+1}\mid a_1$ for all $k\in\mathbb{N}$, and thus the general result by passing to the procreation behaviors of $1_G$ in the state spaces successive images of $\varphi$ with the corresponding restriction of $\varphi$.
\end{proof}

Actually, this is the strongest result on the structure of the nilpotent part which we can derive in general, as the following proposition shows.

\begin{propposition}\label{strongestProp}
For any finite $1$-tree $\Gamma$ whose cycle is a loop and which has rigid procreation such that the procreation behavior of the one vertex on the loop is $(a_k)_{k\in\mathbb{N}^+}$ with $a_m\mid a_n$ for all $n,m\in\mathbb{N}^+$ with $n\geq m$, there exists an FDG $(G,\varphi)$ such that $G$ is abelian and $\Gamma_{\varphi}^{\ast}\cong\Gamma$.
\end{propposition}

\begin{proof}
Consider the finite abelian group $G:=\prod\limits_{i=1}^n{\mathbb{Z}/a_i\mathbb{Z}}=\langle x_1,\ldots,x_n\mid x_ix_j=x_jx_i\hspace{3pt}(i\not=j),x_i^{a_i}=1\hspace{3pt}(i=1,\ldots,n)\rangle$, where $n$ is so large that $a_{n+1}=1$. We specify a nilpotent endomorphism $\varphi$ of $G$ such that the $k$-th kernel of $\varphi$ is the subgroup generated by $x_1,\ldots,x_k$, which is sufficient by Theorem \ref{stateGraphTheo}. $\varphi$ can be defined by specification on the generators $x_i$. We set $\varphi(x_1):=1$ and $\varphi(x_{i+1}):=x_i^{a_i/a_{i+1}}$ for $i=1,\ldots,n-1$. This preserves the orders of generators $x_i$ with $i>1$ and hence defines an endomorphism of $G$. It is clear that any of $x_1,\ldots,x_k$ is mapped to $1_G$ after $k$ applications of $\varphi$, while the other generators \enquote{survive} $k$ applications of $\varphi$.
\end{proof}

\section{An application to finite cyclic groups}

Let us now consider the finite cyclic group $\mathbb{Z}/n\mathbb{Z}$. Any endomorphism of this group is a \enquote{stretch modulo $n$} by a factor $a\in\{0,\ldots,n-1\}$; we denote the corresponding stretch function by $\lambda_a$. FDSs arising from such maps $\lambda_a$ play an important role in pseudorandom number generation (key word: multiplicative congruential generators), and several papers have already been dedicated to the study of their state spaces: Ahmad \cite{Ahm69a} in 1969 investigated the cycle structure of automorphisms of finite cyclic groups. In 2008, Hern{\'a}ndez-Toledo \cite{Her08a} used the structure of the group of units modulo odd prime powers to describe the structure of state spaces of endomorphisms of $\mathbb{Z}/p^k\mathbb{Z}$ for odd primes $p$ as explicitly as possible. He did not treat the case of primary cyclic groups of even order or the general case, though. Sha in his already mentioned paper \cite{Sha12a} investigated state spaces of endomorphisms of general finite cyclic groups, describing, among other things, their graph automorphism groups. Finally, Deng in \cite{Den13a} more generally extensively studied the state spaces arising from affine maps of finite cyclic groups and gave a necessary and sufficient criterion of number-theoretic nature when two such graphs are isomorphic. However, to the author's best knowledge, so far there exists no published explicit formula for the number of isomorphism types of state spaces of endomorphisms of $\mathbb{Z}/n\mathbb{Z}$, which we will now derive as an application of the abstract theory developed in the previous section. To this end, we will extend Hern{\'a}ndez-Toledo's idea of using the structure of the group of units to primary cyclic groups of even order, and the group-theoretic Lemma \ref{coprimeProdLem} will allow us to easily extend our counting formulas from primary cyclic groups to the general case.

To make our text self-comprehensive and since our proof for primary cyclic groups of even order is similar to the one we give for the odd order case, we will prove both cases here. Let us start with the odd order case (note that since $\mathbb{Z}/p^n\mathbb{Z}$ does not decompose as a semidirect product in a nontrivial way, any endomorphism of it is either nilpotent or an automorphism):

\begin{lemma}\label{oddPrimeLem}
Let $p$ be an odd prime, $k\in\mathbb{N}$. Then the number of isomorphism types of state spaces of endomorphisms of $\mathbb{Z}/p^k\mathbb{Z}$ equals $k\cdot(\tau(p-1)+1)$, where $\tau:\mathbb{N}^+\rightarrow\mathbb{N}^+$ denotes the divisor number function. Of these, $k$ correspond to nilpotent endomorphisms and $k\cdot\tau(p-1)$ to automorphisms.
\end{lemma}

\begin{proof}
Because of $\lambda_a^m=\lambda_{a^m\hspace{3pt}(\mathrm{mod}\hspace{3pt}p^k)}$, it is easy to see that $\lambda_a:\mathbb{Z}/p^k\mathbb{Z}\rightarrow\mathbb{Z}/p^k\mathbb{Z}$ is nilpotent if and only if $p\mid a$. Let $v_p^{(k)}(a):=\mathrm{min}\{\nu_p(a),k\}$ denote the \textit{$p$-adic valuation of $a$ modulo $p^k$}; here, $\nu_p(a)$ denotes the usual $p$-adic valuation of $a$, defined as the exponent of the greatest power of $p$ dividing $a$, which is understood to be $\infty$ if $a=0$. If $u\in\{0,\ldots,p-1\}$ with $p\nmid u$, then for all $c\in\{0,\ldots,p-1\}$ and $l\in\{1,\ldots,n\}$, the congruence $up^l\cdot x\equiv c\hspace{3pt}(\mathrm{mod}\hspace{3pt}p^k)$ has the same number of solutions modulo $p^k$ as $p^lx\equiv c\hspace{3pt}(\mathrm{mod}\hspace{3pt}p^k)$ (namely $p^l$ if $v_p^{(k)}(c)\geq l$ and $0$ else) so that the procreation behaviors of the identity element $0$ under $\lambda_{p^l}$ and $\lambda_{up^l}$ are the same and hence their state spaces are isomorphic. So for counting the isomorphism types in the nilpotent case, we only need to consider the endomorphisms $\lambda_{p^l}$ for $l=1,\ldots,n$. But again, by the observations on the solvability modulo $p^k$ of the congruence $p^l\cdot x\equiv c\hspace{3pt}(\mathrm{mod}\hspace{3pt}p^k)$ from above and Theorem \ref{stateGraphTheo}, it is easy to see that the following holds for the procreation behavior in this case: Write $k=q\cdot l+r$ with $q,r\in\mathbb{N}$ and $0\leq r<l$. Then the procreation behavior of the identity under $\lambda_{p^l}$ is $(p^l,p^l,\ldots,p^l,p^r,1,\ldots)$, where the first $q$ procreation numbers are equal to $p^l$. Hence these $k$ nilpotent endomorphisms indeed yield pairwise non-isomorphic state spaces, and we are done in the nilpotent case.

It remains to treat the case $p\nmid a$, where $\lambda_a$ is an automorphism. This is basically the same argumentation as the one of Hern{\'a}ndez-Toledo. We make use of the fact that there is a primitive root $g$ modulo $p^k$, and write $a=g^{u\cdot s\cdot p^l}$, where the numbers $u,t\in\{1,\ldots,p^k-1\}$ with $\mathrm{gcd}(u,p(p-1))=1$, $s$ is a product of powers of the prime divisors of $p-1=p_1^{e_1}\cdots p_r^{e_r}$ (all $e_i\geq 1$), and $l\in\{0,\ldots,k-1\}$. Since the multiplicative order of $g$ modulo $p^m$, $m\in\{1,\ldots,k\}$, is $\phi(p^m)=p^{m-1}(p-1)$, by a basic result of group theory (or elementary number theory), the multiplicative order of $a$ modulo $p^m$ is $$p_1^{e_1-v_{p_1}^{(e_1)}(s)}\cdots p_r^{e_r-v_{p_r}^{(e_r)}(s)}p^{\mathrm{max}\{0,m-1-l\}}.$$ This means that the cycle of any generator of $\mathbb{Z}/p^m\mathbb{Z}$ under the stretch by $a$ has this length, and hence so does the cycle of any element of $\mathbb{Z}/p^k\mathbb{Z}$ of order $p^m$, as $\lambda_a$ on $\mathbb{Z}/p^k\mathbb{Z}$ restricts to an automorphism of the subgroup generated by this element, defining a dynamical structure isomorphic to the one of the corresponding stretch on $\mathbb{Z}/p^m\mathbb{Z}$. We can therefore describe the cycle structure of $\lambda_a$ on $\mathbb{Z}/p^k\mathbb{Z}$ as follows:

It has, in addition to the one trivial fixed point, $p^{l+1}-1$ points (namely the nontrivial elements of the unique subgroup of order $p^{l+1}$) lying on cycles of length $$p_1^{e_1-v_{p_1}^{(e_1)}(s)}\cdots p_r^{e_r-v_{p_r}^{(e_r)}(s)},$$ that is, $$\frac{p^{l+1}-1}{p_1^{e_1-v_{p_1}^{(e_1)}(s)}\cdots p_r^{e_r-v_{p_r}^{(e_r)}(s)}}$$ cycles of that length. Furthermore, for each $j\in\{k+2,\ldots,n\}$, it has $p^j-p^{j-1}$ points (the elements from the complement of the subgroup with $p^{j-1}$ elements in the subgroup with $p^j$ elements) on cycles of length $$p_1^{e_1-v_{p_1}^{(e_1)}(s)}\cdots p_r^{e_r-v_{p_r}^{(e_r)}(s)}p^{j-l-1},$$ i.e., $$p^l\cdot\frac{p-1}{p_1^{e_1-v_{p_1}^{(e_1)}(s)}\cdots p_r^{e_r-v_{p_r}^{(e_r)}(s)}}$$ cycles of that length. From this, we obtain a bijective correspondence between the isomorphism types of state spaces of automorphisms of $\mathbb{Z}/p^k\mathbb{Z}$ and the cartesian product of the set of positive divisors of $p-1$ with the set $\{0,\ldots,k-1\}$, whence there are $k\cdot\tau(p-1)$ such isomorphism types, as we wanted to show.
\end{proof}

The case $p=2$ goes as follows:

\begin{lemma}\label{evenPrimeLem}
Let $k\in\mathbb{N},k\geq 3$. Then the number of isomorphism types of state spaces of endomorphisms of $\mathbb{Z}/2^k\mathbb{Z}$ equals $3k-3$, of which $k$ stem from nilpotent endomorphisms and $2k-3$ from automorphisms. Also, $\mathbb{Z}/2\mathbb{Z}$ has $2$ such isomorphism types, one nilpotent, one periodic, and $\mathbb{Z}/4\mathbb{Z}$ has $4$ isomorphism types, two nilpotent, two periodic.
\end{lemma}

\begin{proof}
This is actually very similar in spirit to the proof of Lemma \ref{oddPrimeLem} which we just gave; the situation is only a bit different because the structure of $(\mathbb{Z}/2^k\mathbb{Z})^{\ast}$ is more complicated compared to the one of $(\mathbb{Z}/p^k\mathbb{Z})^{\ast}$ for odd $p$. However, we still have everything under control. First of all, let us note that as for the nilpotent case, the same argument as for odd $p$ works, so we do not need to discuss it.

As for the periodic case, the cases $k=1,2$ are readily checked separately, and for $k\geq 3$, it is a well-known result of elementary number theory that the group of units $(\mathbb{Z}/2^k\mathbb{Z})^{\ast}$ is not cyclic, but decomposes as a direct product of two cyclic subgroups, one of order $2$ and generated by $-1$, the other of order $2^{k-2}$, generated by $5$. So write $a=(-1)^{\epsilon}5^{u\cdot 2^l}$ in $\mathbb{Z}/2^k\mathbb{Z}$ with $\epsilon\in\{0,1\}$, $u\in\{1,\ldots,2^{k-2}-1\}$ odd and $l\in\{0,\ldots,k-2\}$. Apparently, the multiplicative order of $a$ modulo $2^m$ with $m\geq 2$ then is $2^{\mathrm{max}\{\epsilon,m-2-l\}}$, and modulo $2$ it is just $1$. So we have two certain fixed points in this case (the identity and the uniquely determined element of additive order $2$), and additionally, for all $m\in\{2,\ldots,k\}$, we have $\frac{2^{m-1}}{2^{\mathrm{max}\{\epsilon,m-2-l\}}}$ cycles of length $2^{\mathrm{max}\{\epsilon,m-2-l\}}$.

Hence different values for $\epsilon$ give non-isomorphic state spaces (because there will be more than two fixed points if and only if $\epsilon=0$), but also clearly, for a fixed value of $\epsilon$ and varying values of $l\in\{0,\ldots,k-2\}$, we also get pairwise non-isomorphic state spaces, \textit{except} for $\epsilon=1$ and $l=k-3,k-2$ (which yield isomorpic state spaces), whereas different choices for $u$ never have any influence on the isomorphism type. Hence in this case, there are $2\cdot (k-1)-1=2k-3$ isomorphism types of automorphism state spaces.
\end{proof}

Now, for counting isomorphism types of state spaces of endomorphisms, it is not difficult to generalize from the primary cyclic case to arbitrary finite cyclic groups by using the following observation:

\begin{lemma}\label{coprimeProdLem}
Let $(G_1,\psi_1), (G_1,\psi_1'), (G_2,\psi_2)$ and $(G_2,\psi_2')$ be FDGs such that $\mathrm{gcd}(|G_1|,|G_2|)=1$. Then if $\Gamma_{\psi_1\times\psi_2}\cong\Gamma_{\psi_1'\times\psi_2'}$, then $\Gamma_{\psi_1}\cong\Gamma_{\psi_1'}$ and $\Gamma_{\psi_2}\cong\Gamma_{\psi_2'}$.
\end{lemma}

\begin{proof}
It suffices to show that the procreation behavior of the identity in $\Gamma_{\psi_1\times\psi_2}^{\ast}$ and the various orders of periodic point subgroups $\mathrm{per}_n(\psi_1\times\psi_2)$ uniquely determine the corresponding parameters in $\Gamma_{\psi_1}$ and $\Gamma_{\psi_2}$.

As for the procreation behavior, let $a_n(\psi)$ for $n\in\mathbb{N}^+$ and $\psi$ an endomorphism of a finite group $G$ denote the $n$-th procreation number of $1_G$, i.e., by Theorem \ref{stateGraphTheo}, the index $[\mathrm{ker}^{(n)}(\psi):\mathrm{ker}^{(n-1)}(\psi)]$. It is clear that $\mathrm{ker}^{(n)}(\psi_1\times\psi_2)=\mathrm{ker}^{(n)}(\psi_1)\times\mathrm{ker}^{(n)}(\psi_2)$, so that $a_n(\psi_1\times\psi_2)=a_n(\psi_1)\cdot a_n(\psi_2)$. But since $a_n(\psi_i)\mid |G_i|$, by the coprimality assumption, we can read off the values of $a_n(\psi_1)$ and $a_n(\psi_2)$ from this product. The argumentation for the orders of periodic point subgroups is similar, using $\mathrm{per}_n(\psi_1\times\psi_2)=\mathrm{per}_n(\psi_1)\times\mathrm{per}_n(\psi_2)$.
\end{proof}

In view of the Chinese Remainder Theorem and the fact that $\psi_1\times\psi_2$ is an automorphism (resp.~nilpotent) if and only if $\psi_1$ and $\psi_2$ have the corresponding property, this yields:

\begin{theorem}\label{cyclicCountTheo}
Let $n=2^k\cdot p_1^{k_1}\cdots p_l^{k_l}$ be a positive natural number with the prime factor decomposition displayed such that $k\geq0$, $l\geq0$ and $k_1,\ldots,k_l\geq1$. Furthermore, let $\tau:\mathbb{N}^+\rightarrow\mathbb{N}^+$ denote the divisor number function and let $$\delta_{[k\leq 2]}:=\left\{\begin{array}{cl}1, & \text{if}\hspace{3pt}k\leq 2, \\ 0, & \text{else}\end{array}\right..$$ Then the number of isomorphism types of state spaces of endomorphisms of $\mathbb{Z}/n\mathbb{Z}$ is precisely

$$\mathrm{max}\{2^{k\cdot\delta_{[k\leq 2]}},3k-3\}\cdot\prod\limits_{i=1}^l{k_i(\tau(p_i-1)+1)}.$$

Of these, precisely $$\mathrm{max}\{1,k\}\cdot\prod\limits_{i=1}^l{k_i}$$ are $1$-trees, and precisely $$\mathrm{max}\{1,k,2k-3\}\cdot\prod\limits_{i=1}^l{k_i\tau(p_i-1)}$$ are disjoint unions of cycles.\qed
\end{theorem}

\end{document}